\documentclass[10pt]{article}
\usepackage{amsmath,amssymb}

\newtheorem{theorem}{Theorem}[section]
\newtheorem{lemma}[theorem]{Lemma}
\newtheorem{proposition}[theorem]{Proposition}

\newcommand{\qed}{\enspace\vrule height6pt width4pt depth2pt}
\newenvironment{proof}{\par\noindent{\bf Proof.}}{$\qed$\par\bigskip}

\newcommand{\N}{{\mathbb N}}
\newcommand{\Z}{{\mathbb Z}}

\newcommand{\U}{\operatorname{U}}

\newcommand{\supp}{\operatorname{supp}}

\newcommand{\Cl}{\operatorname{cl}}
\newcommand{\gr}{\mbox{gr}}
\newcommand{\FaM}{\operatorname{FaM}}
\newcommand{\ssupp}{\operatorname{Hsupp}}
\newcommand{\Fa}{\operatorname{Fa}}

\begin{document}

\date{}
\title{Normal domains with monomial presentations}
\author{Isabel Goffa, Eric Jespers, Jan Okni\'{n}ski}
\maketitle

\begin{abstract}
Let $A$ be a finitely generated commutative algebra
over a field $K$ with a presentation $A= K\langle
X_{1},\ldots ,X_{n} \mid R\rangle$, where $R$ is a
set of monomial relations in the generators
$X_{1},\ldots , X_{n}$.  So $A=K[S]$, the semigroup
algebra of the monoid $S=\langle X_{1},\ldots ,X_{n}
\mid R\rangle$. We characterize, purely in terms of
the defining relations, when $A$ is an integrally
closed domain, provided $R$ contains at most two
relations. Also the class group of such algebras $A$
is calculated.
\end{abstract}

\noindent {\it Mathematics Subject Classification 2000:} primary
16S36,
13B22; secondary 14M25, 16H05, 13C20, 20M14 \\
{\it keywords:} normal domain, class group, finitely presented
algebra, semigroup algebra, commutative semigroup, normal
semigroup


\section{Introduction}

Normal Noetherian domains,  also called integrally closed
Noetherian domains, are of fundamental importance in several areas
of mathematics. In the literature one can find several concrete
constructions of such rings that are algebras over a field $K$ and
that have a presentation in which the relations are of monomial
type. Such algebras are commutative semigroup algebras $K[S]$ of a
finitely generated abelian and cancellative monoid $S$ (that is,
$S$ is a submonoid of a finitely generated abelian group $G$).
Within the context of commutative ring theory, these algebras
received a lot of attention (see for example
\cite{bruns-herzog,gil}). We recall some well known facts. First,
a commutative semigroup algebra $K[S]$ of a monoid $S$ is
Noetherian if and only if $S$ is finitely generated. In this case
$K[S]$ also is finitely presented. Second, $K[S]$ is a domain if
and only if $S$ is a submonoid of a torsion free abelian group.
Recall that an affine semigroup $S$ is a finitely generated
submonoid of a free abelian group. If, moreover, the unit groups
$\U (S)$ is trivial, that is  $\U (S)=\{ 1\}$, then $S$ is said to
be positive. Third (see \cite[Proposition~6.1.4]{bruns-herzog} or
\cite[Proposition~13.5]{sturm}), if $M$ is an affine monoid then
$K[M]$ is normal if and only if $M$ is normal (i.e. if $g\in
MM^{-1}$, the group of fractions of $M$, and $g^{n}\in M$  for
some $n\geq 1$ then $g\in M$). Moreover, such monoids $M$ are
precisely the monoids of the form $U\times M'$, where $U$ is a
finitely generated free abelian group and $M'$ is a positive
monoid so that $M'=(M')(M')^{-1}\cap F^{+}$, with $F^{+}$ the
positive cone of a free abelian group $F$. Note that if $M$ is
positive and of rank $n$, that is $MM^{-1}$ is a group of torsion
free rank $n$, then $M$ is  isomorphic to a submonoid of ${\mathbb
N}^{n}$, a free abelian monoid of rank $n$. So, normality of
$K[S]$ is a homogeneous property, i.e., a condition on the monoid
$S$. This was one of the motivating reasons for these
investigations. Furthermore, it is well known that $\Cl (K[S])$,
the class group of $K[S]$, is naturally isomorphic with $\Cl (S)$,
the class group of $S$ (see for example
\cite[Theorem~2.3.1]{brunsbook}). As an application one obtains
much easier calculations for the class group of several classical
examples of Noetherian normal domains.  So the study of normal
positive monoids is relevant in the context of number theory.
Another reason for their importance is  their connection to
geometry, especially in the context of toric varieties and convex
polytopes (see for example
\cite{brunsbook,miller-sturm,sturm,villareal} for an extensive
bibliography of the subject, its computational aspects and
applications to other fields).

The study of the above problems is also crucial in a
noncommutative setting. Indeed, noncommutative
maximal orders of the form $K[S]$, with $S$ a
cancellative nonabelian monoid, appear in the search
of set-theoretical solutions of the quantum
Yang-Baxter equation. Gateva-Ivanova and Van den
Bergh \cite{GIVdB} and Etingof, Schedler and
Soloviev in \cite{ESS} showed that such solutions
are determined by monoids $M$ of $I$-type. In
\cite{jes-gof} this was extended to the larger class
of monoids of $IG$-type. Such monoids are contained
in a finitely generated abelian-by-finite group and
their algebras share many properties with
commutative polynomial algebras. In particular, they
are maximal orders in a division algebra and the
algebraic structure of $M$ is determined by a normal
positive submonoid and a finite solvable group
acting on it. More generally, as shown in \cite{PI},
every prime maximal order $K[S]$ satisfying a
polynomial identity is in some sense built on the
basis of a normal abelian submonoid of $S$ and every
abelian normal monoid can be used to construct a
family of noncommutative maximal orders. For more
details on noncommutative orders  we refer the
reader to \cite{jes-okn-book}.

In this paper we deal with Noetherian commutative
semigroup algebras $K[S]$ that are defined by at
most two monomial relations. We obtain a
characterization  purely in terms of the defining
relations, of when such an algebra is a normal
domain. It is easily seen that if $K[S]$ is such an
algebra then $S$ has codimension at most $2$. Recall
that $S$ has codimension $n-d$ if it is generated by
$n$ elements and $S\subseteq {\mathbb N}^{d}$.
Recently Dueck, Ho\c{s}ten and Sturmfels obtained
necessary conditions for such algebras to be normal.
In order to state this we recall that given a term
order $\prec$ on the free abelian monoid $F=\langle
u_{1},\ldots, u_{n}\rangle $, the initial ideal
$I_{\prec}$ of $S$ (corresponding to this order) is
the ideal of $F$ consisting of all leading (highest)
monomials in every relation that holds in $S$.

\begin{proposition} (\cite[Theorem~1]{dueck})
\label{dueck} Suppose $S$ is a positive monoid of codimension two.
If $S$ is normal then $S$ has a square free initial ideal (that
is, a semiprime ideal in $S$).
\end{proposition}

If, moreover, $S$ is a homogeneous monoid (that is, $S$ is defined
by relations that are homogeneous with respect to the total
degree) then the converse follows from Proposition~13.15 in
\cite{sturm}. The latter  says that if $S$ is a homogeneous
submonoid of ${\mathbb Z}^{d}$ such that for some order $\prec$
the corresponding initial ideal $I_{\prec}$ is square free, then
$S$ is a normal monoid. Theorem~2 in \cite{dueck} also says that
if $S$ is a positive monoid of codimension $n-d$ then there is an
algorithm to decide whether $A$ is normal, whose running time is
polynomial.

From the characterization proved in this paper it follows that the
converse of Proposition~\ref{dueck} holds for an arbitrary
positive monoid $S$ defined by at most two relations (so without
the homogeneous assumption). Exercise~13.17 in \cite{sturm}
implies that this converse is false in general.  It is worth
mentioning that other constraints for normality of abelian monoids
have been studied in \cite{ohsugi,rosales,simis}.

As an application, we determine the class group $\Cl
(S)$ in terms of the combinatorial data contained in
the defining relations.

\section{One-relator monoids}

Our main aim is to describe when a positive monoid defined by at
most two relations is normal. A first important obstacle to
overcome is to determine when such monoids are cancellative, i.e.,
when they are contained in a group  and next to decide when this
group can be assumed torsion free. Because of the comments given
in the introduction, and since we are mainly interested in such
monoids that are normal, we only need to deal with monoids $S$ so
that $\U(S) = \{1\}$. In this context we mention that in
\cite{ros} an algorithm of Contejean and Devie is used to
determine whether a finitely generated monoid given by a
presentation is cancellative.

We will use the following notation. By $\FaM_{n}$ we denote a free
abelian monoid of rank $n$. If $\FaM_{n}=\langle u_{1},\ldots ,
u_{n}\rangle$ and $w = u_{1}^{a_{1}}\cdots u_{n}^{a_{n}} \in
\FaM_{n}$, then put $\supp(w) = \{u_{i} \mid a_{i} \neq 0 \}$, the
support of $w$, and $\ssupp(w) = \{u_{j}\mid a_{j}>1\}$. We say
that $w$ is square free if $\ssupp (w)=\emptyset$. Now, suppose
$S$ has a presentation $$S = \langle u_{1},\ldots ,u_{n} \mid
w_{1} = v_{1},\ldots ,w_{m} = v_{m}\rangle,$$ where $w_{i}, v_{i}$
are nonempty words in the free abelian monoid $\FaM_{n}= \langle
u_{1},\ldots ,u_{n} \rangle$. Clearly,  $\U (S)=\{ 1\}$ and if $S$
is cancellative, then we may assume it has a presentation with
$$\supp(w_{i}) \cap \supp(v_{i}) = \emptyset,$$ for all $i$.

Recall from Lemma~6.1 in \cite{ohsugi} that if
$K[S]$ is a normal domain and $w_{i}=v_{i}$ is
independent of the other defining relations then
at least one of $w_{i}$ or $v_{i}$ is square
free.

\begin{proposition} \label{single}
Let $S$ be an abelian monoid defined by the
presentation $$\langle u_{1},\ldots ,u_{n}\mid
u_{1}\cdots u_{k} = u_{k+1}^{a_{k+1}} \cdots
u_{n}^{a_{n}}\rangle $$ for some positive integers
$a_{k+1},\ldots, a_{n}$ and some $k<n$. Let
$\FaM_{k(n-k)}=\langle x_{i,j} \mid 1\leq i \leq k,
\; 1\leq j \leq n-k\rangle$, a free abelian monoid
of rank $k(n-k)$. For $1\leq j\leq k$ put
 $$v_{j}=x_{j,1}^{a_{k+1}}x_{j,2}^{a_{k+2}}  \cdots
x_{j,n-k}^{a_{n}},$$  and for  $k+1 \leq j \leq n$
put
 $$v_{j}=x_{1,j-k}x_{2,j-k}\cdots x_{k,j-k}.$$
Then $S\cong V=\langle v_{1},\ldots, v_{n}\rangle
\subseteq \FaM_{k(n-k)}$ (in particular, $S$ is
cancellative) and $\langle v_{1},\ldots ,
v_{n}\rangle$ has $v_{1}\cdots
v_{k}=v_{k+1}^{a_{k+1}}\cdots v_{n}^{a_{n}}$ as its
only defining relation.
\end{proposition}
\begin{proof} Let $V=\langle v_{1},\ldots , v_{n}\rangle
\subseteq \FaM_{k(n-k)}$. Clearly, $v_{1}\cdots
v_{k} = v_{k+1}^{a_{k+1}} \cdots v_{n}^{a_{n}}$ and
thus $V=\langle v_{1},\ldots,v_{n}\rangle$ is a
natural homomorphic image of $S$.

Since all $a_{i}\neq 0$, it is easy to see that
every relation in $V$ (with disjoint supports with
respect to the $v_{i}$'s)  must involve all
generators $v_{i}$. Moreover, since $v_{1},\;
v_{k+1}$ are the only generators involving
$x_{1,1}$, it follows that in such a relation
$v_{1},\; v_{k+1}$ are on opposite sides of the
equality. And also $v_{k+2},\ldots ,v_{n}$ must be
on the side opposite to $v_{1}$ (look at the
appearance of $x_{1,2},\, x_{1,3},\, \ldots ,
x_{1,n-k}$ in order to see this).  Similarly, by
looking at the appearance of $x_{21},x_{31},\ldots,
x_{k1}$, we get that $v_{2},\ldots ,v_{k}$ must be
on the side opposite to $v_{k+1}$. It follows that
every relation in $V$, possibly after cancellation,
must be of the form
\begin{eqnarray}\label{rel} v_{1}^{c_{1}}\cdots
v_{k}^{c_{k}} = v_{k+1}^{c_{k+1}} \cdots
v_{n}^{c_{n}}
\end{eqnarray}
for some positive integers $c_{j}$. Again, using the
fact that $x_{i,j}$'s are independent and comparing
the exponent of $x_{i,j}$ on both sides of
(\ref{rel}), we get that $a_{k+j}c_{i}=c_{k+j}$ for
$1\leq i \leq k$ and $j=1,2,\ldots ,n-k$. This
implies that $c_{1}=c_{2}=\cdots =c_{k}$. Hence
relation (\ref{rel}) is of the form $(v_{1}\cdots
v_{k})^{c_{1}}=(v_{k+1}^{a_{k+1}}\cdots
v_{n}^{a_{n}})^{c_{1}} $. So it is a consequence of
the relation defining $S$ with every $u_{j}$
replaced by $v_{j}$. It follows that $V\cong S$.
\end{proof}

Note that one can verify that the monoid $V$, as described in the
previous proposition, is such that $V=VV^{-1}\cap \FaM_{k(n-k)}$.
So, by the comments given in the introduction, $V$ is normal.
Alternatively, it easily follows from the defining relation that
$S=\bigcup_{1\leq i \leq k} F_{i}$, with $F_{i}= \langle u_{j}
\mid 1\leq j\leq n,\; j\neq i\rangle$ a free abelian monoid with
group of quotients $SS^{-1}$. Since each $F_{i}$ is normal we thus
obtain that $S$ is normal as well
(\cite[Proposition~3.1.1]{brunsbook}).

 Hence, the Proposition~\ref{single} and its preceding comment
yield at once a description one-relator positive monoids that are
normal.

\begin{proposition}\label{1relation}
Let $S$ be the abelian monoid defined by the
presentation $$\langle u_{1},\ldots ,u_{n}\mid
w_{1}=w_{2}\rangle , $$ with nonempty words $w_{1}=
u_{1}^{a_{1}}\cdots u_{k}^{a_{k}}$, $w_{2} =
u_{k+1}^{a_{k+1}} \cdots u_{n}^{a_{n}}$, where
$k<n$, and each $a_{i}$ is a nonnegative integer.
The following conditions are equivalent.
\begin{enumerate}
\item The
semigroup $S$ is a normal positive monoid, normal
(or equivalently, the semigroup algebra $K[S]$ is
a normal domain).
\item  $\ssupp(w_{1}) = \emptyset$ or $\ssupp(w_{2}) =
\emptyset$.
\end{enumerate}
\end{proposition}

In the remainder of  this section we describe the class group
$\Cl(S)$ of a  one-relator normal positive monoid $S$. For
convenience sake we recall some terminology for an affine normal
monoid $M$ (see \cite{cho,gil}; at an algebra level we refer to
\cite{fossum}). For a subset $I$ of $MM^{-1}$ we put $(M:I)=\{
g\in MM^{-1} \mid gI\subseteq M\}$. A fractional ideal $I$ of $M$
is a subset of $MM^{-1}$ so that $MI\subseteq I$ and $mI\subseteq
M$ for some $m\in M$. A fractional ideal is said to be divisorial
if $I=I^{*}$, where $I^{*}=(M:(M:I))$. The set of all divisorial
fractional ideals is denoted by $D(M)$. It is a free abelian group
for the divisorial product $I*J=(IJ)^{*}$, for $I,J\in D(M)$, with
basis the set of minimal prime ideals. Also,  $M=\bigcap M_{P}$,
where the intersection runs over all minimal primes of $M$, and
all localizations $M_{P}$ are discrete valuation monoids (see for
example \cite{cho,gil}). Furthermore, for an ideal $I$ of $M$ one
has, in the divisorial group $D(M)$, that $I=(\prod_{P}
P^{n(P)})^{*}$ if and only if $M_{P}I=M_{P}P^{n(P)}$, with all
$n_{P}\geq 0$. Moreover, $n_{P}>0$ if and only if $I\subseteq P$.

By definition $\Cl(M)=D(M)/P(M)$, where $P(M)=\{ Mg\mid g\in
MM^{-1}\}$.

Let $S$ be again as in Proposition~\ref{1relation}. We will use
the same notation for the generators $u_{i}$ of the free monoid
$\FaM_{n}$ and for their images in $S$, if unambiguous. So, every
$S_{u_{j}}$ in $D(S)$ is a (unique) product of the minimal primes
of $S$. In the following lemma we compute these decompositions
provided all $a_{i}$ are positive integers. Clearly, in this case,
the minimal primes of $S$ are the ideals $P_{yz}$ generated by the
set $\{ u_{y},u_{z}\}$, where $y \in \{1,\ldots , k\}$, $z\in
\{k+1, \ldots , n\}$.

\begin{lemma}\label{prin}
Let $S = \langle u_{1},\ldots , u_{n}\mid u_{1}\cdots u_{k} =
u_{k+1}^{a_{k+1}}\cdots u_{n}^{a_{n}}\rangle$ be a normal monoid,
 with all $a_{i}\geq 1$, and let $P_{yz}$ denote the minimal
prime ideal of $S$ that is generated by the set $\{
u_{y},u_{z}\}$, where $y \in \{1,\ldots , k\}$, $z\in \{k+1,
\ldots , n\}$. Then $$Su_{z} = P_{1z} \ast \cdots \ast P_{kz}
\quad \mbox{and} \quad Su_{y} = P_{yk+1}^{a_{k+1}} \ast \cdots
\ast P_{yn}^{a_{n}},$$ for  $z\in \{k+1, \ldots , n\}$ and  $y \in
\{1,\ldots , k\}$.
\end{lemma}
\begin{proof}
First, let $y\in \{ 1,\ldots ,k\}$. Note that the only minimal
primes containing $u_{y}$ are $P_{y,z}$, with $z\in  \{ k+1,\ldots
, n\}$. Hence $Su_{y}=\left( \prod_{k\leq z\leq n}
P_{y,z}^{e(z)}\right)^{*}$, with $e(z)\geq 1$. Furthermore, in the
localization $T=S_{P_{y,z}}$ we have that $u_{i},u_{j}$ are
invertible for $y\neq i\in \{ 1,\ldots , k\}$ and $z\neq j\in \{
k+1,\ldots , n\}$. Hence, from the defining relation it follows
that $Tu_{y}=Tu_{z}^{a_{z}}$ and thus also
$Tu_{y}=TP_{y,z}^{a_{z}}$. Consequently, $e(z)=a_{z}$ and thus
$Su_{y} =P_{yk+1}^{a_{k+1}} \ast \cdots \ast P_{yn}^{a_{n}}$, as
desired.

Second, assume $z\in \{k+1,\ldots ,n\}$. Then, for any $y\in \{
1,\ldots , k\}$, it is easily seen from the defining relation that
$Tu_{y}\subseteq Tu_{z}$, with $T=S_{P_{y,z}}$. Thus
$Tu_{z}=TP_{y,z}$. Therefore, as above, $Su_{z}=P_{1z} \ast \cdots
\ast P_{kz}$.
\end{proof}

\begin{theorem}\label{1main}
Let $S = \langle u_{1},\ldots , u_{n},\ldots , u_{m}
\mid u_{1}\cdots u_{k} = u_{k+1}^{a_{k+1}}\cdots
u_{n}^{a_{n}}\rangle$ be a positive normal monoid
(with all $a_{i}\geq 1$ and $n\leq m$). Then
$$\Cl(K[S])\cong \Cl(S) \cong
\Z^{k(n-k)-(n-1)}\times (\Z_{d})^{k-1},$$ where $d=\gcd (a_{k+1},
\ldots ,a_{n})$, $k(n-k)$ is the number of minimal primes in $S$
not containing one of the independent generators $u_{n+1},\ldots ,
u_{m}$, and $m-1$ is the torsion free rank of $SS^{-1}$.
\end{theorem}
\begin{proof}
Clearly, $S=S'\times \FaM_{m-n}$, where $\FaM_{r}=\langle
a_{n+1},\ldots , a_{m}\rangle$ is a free abelian monoid, and
$S'=\langle  u_{1},\ldots , u_{n} \mid u_{1}\cdots u_{k} =
u_{k+1}^{a_{k+1}}\cdots u_{n}^{a_{n}} \rangle $. So, $S$ is normal
if and only if $S'$ is normal. Because also   $\Cl (S')=\Cl(S)$,
we may assume $S=S'$.

Clearly,  the result is true for $k=1$. So assume
that $k\geq 2$. As there are $k(n-k)$ minimal primes
$P_{yz}$ in $S$ (with $1 \leq y \leq k$, $k+1 \leq z
\leq n$), we get that $D(S) \cong \Z^{k(n-k)}$. On
the other hand, $P(S) = \gr (Su_{i} \mid i=1, \ldots
, n)$. By Lemma~\ref{prin},
$Su_{j}=\left(\prod_{l=1}^{k} P_{lj}\right)^{*} ,
Su_{i}=\left(\prod_{l=k+1}^{n}
P_{il}^{a_{l}}\right)^{*}$ for $i \in \{1, \ldots ,
k\}, j \in \{k+1,\ldots , n\}$. We consider
\begin{eqnarray*}
\Cl(S) =\gr(P_{yz} \mid y \in \{1, \ldots , k\}, \; z \in
\{k+1,\ldots , n\})  / \gr(Su_{i} \mid 1 \leq i \leq n).
\end{eqnarray*}
as a finitely generated $\mathbb Z$-module. So  its
presentation corresponds to an integer matrix $M$ of
size $k(n-k) \times n$. The rows of $M$ are indexed
by elements of the set $R=\{ (i,j) \mid i\in
\{1,\ldots ,k\}, j\in \{ k+1,\ldots, n\} \}$. We
agree on the lexicographic ordering of the set of
rows of $M$. The columns are indexed by $C=\{
1,2,\ldots ,n\}$, where the $i$-th column
corresponds to the generator $Su_{i}$, written as a
vector in terms of the minimal primes of $S$.

We consider the block decomposition of $M$ determined by the
following partitions of the sets $C$ and $R$ of columns and rows:
$C=D_{1}\cup D_{2}$, where $D_{1}=\{1,\ldots ,k\}$ and $D_{2}=\{
k+1,\ldots ,n\}$ and $R=R_{1}\cup \cdots \cup R_{k}$, where
$R_{i}=\{ (i,j)\mid j=k+1,\ldots ,n\}$. Then $M$ has the following
form: $$\left(
\begin{array}{ccccccccc} a_{k+1} & 0 &  0 &  \cdots
& 0 & 1 & 0 &\cdots  & 0 \\ a_{k+2}  & 0 & 0 & \cdots  & 0 & 0 & 1
& \cdots & 0 \\ \cdots & \cdots & \cdots &\cdots  & \cdots  &
\cdots &\cdots &\cdots & \cdots \\ a_{n}    & 0 &  0 & \cdots & 0
& 0 & 0 & \cdots &  1\\ 0 &  a_{k+1} & 0 & \cdots  & 0 & 1 & 0 &
\cdots & 0 \\ 0 & a_{k+2}  & 0 & \cdots  & 0 & 0 & 1 & \cdots & 0
\\ \cdots  & \cdots &\cdots
&\cdots  &\cdots &\cdots &\cdots &\cdots & \cdots \\
0 & a_{n}  &0   &  \cdots  & 0 & 0 & \cdots  & \cdots  &  1\\
\cdots & \cdots & \cdots & \cdots  & \cdots  & \cdots &\cdots
&\cdots & \cdots \\ \cdots & \cdots & \cdots & \cdots  & \cdots  &
\cdots &\cdots &\cdots & \cdots \\ \cdots & \cdots & \cdots &
\cdots  & \cdots  & \cdots &\cdots &\cdots & \cdots \\ \cdots &
\cdots & \cdots & \cdots  & \cdots  & \cdots &\cdots &\cdots &
\cdots \\ 0 &
\cdots &0 & \cdots &  a_{k+1} & 1 & 0 & \cdots & 0 \\
 0 &  \cdots  & 0 & \cdots &  a_{k+2}  & 0 & 1 & \cdots & 0 \\
\cdots & \cdots & \cdots &\cdots  & \cdots  & \cdots &\cdots &\cdots & \cdots \\
0 & \cdots & 0 & \cdots   &  a_{n}    & 0 & 0 &  \cdots   &  1\\
\end{array} \right) $$
We subtract the subsequent rows of the last row block $R_{k}$ from
the corresponding rows of all other row blocks. Then from column
$k$ we subtract $\sum _{i=k+1}^{n}a_{i}C_{i}$, where $C_{i}$
denotes the $i$-the column. The obtained matrix $M'$ has the
$(R_{k},C)$-block of the form $M'_{R_{k},C} = (0,I)$, where $I$ is
the $(n-k)\times (n-k)$ identity matrix and $M'_{R_{i}D_{2}}$ is a
zero matrix for every $i\neq k$. Let $T=R\setminus R_{k}$. The
last column of the submatrix $M'_{T,D_{1}}$ has the form
$(-a_{k+1},\ldots ,-a_{n}, \ldots ,-a_{k+1},\ldots ,-a_{n})^{t}$,
hence adding all other columns of $M'_{T,D_{1}}$ to it, we get a
matrix $N$ such that $N=M_{T,D_{1}}$. Clearly, the normal form of
$N$ involves $k-1$ entries equal to $d=\gcd (a_{k+1},\ldots
,a_{n})$ and no other nonzero entries. The result follows.
\end{proof}

\section{Two-relator monoids}

In this section we obtain a characterization of normal positive
monoids that are defined by two relations. The class group of such
monoids $S$, and therefore of the corresponding algebras $K[S]$,
is also determined.

\begin{theorem}\label{2relations}
Let $S = \langle u_{1},\ldots ,u_{n} \rangle$ be a
finitely presented abelian monoid with independent
defining relations $w_{1} = w_{2}$ and $w_{3} =
w_{4}$ and, $\left| \supp(w_{i}) \right| \geq 1$ for
all $i$. The following conditions are equivalent.
\begin{enumerate}
\item  The semigroup
$S$ is a normal positive monoid (or equivalently,
the semigroup algebra $K[S]$ is a normal domain).
\item $S$ is a positive monoid with an initial ideal $I_{\prec}$
of $S$ that is square free.
\item The
following conditions hold:
\begin{enumerate}
\item $\supp(w_{1}) \cap \supp(w_{2}) = \emptyset$,
$\supp(w_{3}) \cap \supp(w_{4}) = \emptyset$,
\item $\ssupp(w_{1}) = \emptyset$ or $\ssupp(w_{2}) = \emptyset$,
\item $\ssupp(w_{3}) = \emptyset$ or $\ssupp(w_{4}) = \emptyset$,
\item
if there exist $i \in \{1,2\}$, $j \in \{3,4\}$ such
that $\supp(w_{i}) \cap \supp(w_{j}) \neq
\emptyset$, then one of the following properties
holds (we may assume for simplicity that $i=1$ and
$j=3$):
\begin{itemize}
\item
$\supp(w_{k}) \cap \supp(w_{l}) = \emptyset$ for all
pairs $\{k,l\} \neq \{1,3\}$  with $k\neq l$, and
$\ssupp(w_{2}) = \emptyset$ or $\ssupp(w_{4})=
\emptyset$,
\item there exists a pair $k\neq l$ such that $\{2,4\}\neq \{k,l\} \neq \{1,3\}$
and $\supp(w_{k}) \cap \supp(w_{l}) \neq \emptyset$
(for simplicity assume $k=2$, $l=3$), $\supp(w_{4})
\cap \supp(w_{i}) = \emptyset$ for $i = 1,2,3$ and
$\ssupp(w_{4}) = \emptyset$.
\end{itemize}
\end{enumerate}
\end{enumerate}
\end{theorem}
\begin{proof}
Note that $S = S_{1} \times S_{2}$, where $S_{2}$ is
the free abelian monoid generated by
$$\{u_{1},\ldots ,u_{n}\} \setminus
(\bigcup_{i=1}^{4} \supp(w_{i}))$$ and $$S_{1} = \langle
\bigcup_{i=1}^{4} \supp(w_{i})\rangle.$$ Since $S_{2}$ is a normal
positive monoid, it follows that $S$ is a normal positive monoid
if and only if $S_{1}$ is such a monoid, i.e. we may assume that
$\{u_{1},\ldots ,u_{n}\} = \bigcup_{i=1}^{4} \supp(w_{i})$.

It follows from Proposition~\ref{dueck} that (1) implies (2). We
now prove (2) implies (3). So assume that $I=I_{\prec}$ is a
square free ideal for some term order $\prec$  and $S$ is a
positive monoid. In order to prove (3.a) suppose for example that
$\supp (w_{1})\cap \supp (w_{2}) \neq \emptyset$. Then write
$w_{1}=uw_{1}', w_{2}=uw_{2}'$ for a nontrivial word $u$ and some
$w_{1}',w_{2}'$ such that $\supp(w_{1}')\cap
\supp(w_{2})'=\emptyset$. Hence, in $S$, we have $w_{1}'=w_{2}'$,
and thus each of $w_{1}',w_{2}'$ is divisible by some of the
$w_{j}$'s. So, by symmetry, we may assume that $w_{1}'=w_{3}z$ and
$w_{2}'=w_{4}y$. Let $m$ be the maximal positive integer such that
$w_{1}'=w_{3}^{m}z'$ and $w_{2}'=w_{4}'y'$ for some $z',y'$. Then
$z'=y'$ holds in $S$ and it follows that $z'=y'$ as words
(otherwise $y',z'$ would be again divisible by $w_{3},w_{4}$,
respectively, contradicting the choice of $m$). It follows that
the relation $w_{1}=w_{2}$ is a consequence of $w_{3}=w_{4}$, a
contradiction. So (3.a) follows.

In order to prove conditions (3.b),(3.c) and (3.d)
we introduce the following notation. For a word $w$
in $u_{1},\ldots ,u_{n}$ we define
$\sqrt{w}=x_{1}\cdots x_{p}$ where $\supp (w)=\{
x_{1},\ldots ,x_{p}\}$.

Note that if $\supp (w_{1})\cap \supp (w_{3}) \neq
\emptyset$ then we must have that $\supp (w_{2})\cap
\supp (w_{4}) = \emptyset$. Indeed, for otherwise,
the ideal $K[S](w_1-w_{2},\; w_{3}-w_{4}) \subseteq
K[S](u_{i},u_{j})$, for some $i\neq j$. Since both
ideals are  height two primes, they must be equal, a
contradiction (note that $S$ is, by assumption, a
positive monoid and thus $K[S]$ is a domain). If
$\supp (w_{2})\cap \supp (w_{3}) \neq \emptyset$
then, by the same reasoning, $\supp (w_{1})\cap
\supp (w_{4}) = \emptyset$. Hence we have shown that
either all $\supp (w_{i})$ are disjoint or
$\supp(w_{i})\cap \supp (w_{j})\neq \emptyset $ for
exactly one pair $i,j$ or this holds for exactly two
pairs and these pairs are of the form $i,j$ and
$i,m$ for some $i,j,m$. So, by symmetry, it is
enough to deal with the three cases considered
below.

If all $\supp (w_{i})$ are disjoint then let for example
$w_{2}\prec w_{1}$ and $w_{4}\prec w_{3}$. It  easily follows from
the assumption that $w_{1},w_{3}$ must be square free and hence
(3.b),(3.c) and (3.d) hold.

Next, assume that $\supp (w_{1}) \cap \supp
(w_{3})\neq \emptyset$ and $\supp (w_{i}) \cap \supp
(w_{j}) =\emptyset$ for every pair $(i,j)\neq
(1,3)$.

To prove (3.d) we need to show that $\ssupp
(w_{i})=\emptyset$ for $i=2$ or $i=4$. So, suppose
otherwise, that is, $w_{2},w_{4}$ are not square
free. Then $w_{1},w_{3}\in I$ and $w_{2}\prec w_{1},
w_{4}\prec w_{3}$  (because for example if
$w_{1}\prec w_{2}$ then $w_{2}\in I$, so $w_{2}\neq
\sqrt{w_{2}}\in I$, whence $\sqrt{w_{2}}$ is in a
nontrivial relation in $S$, but it cannot be
divisible by any of the words $w_{i}, i=1,2,3,4$, a
contradiction). Let $w_{k}=ww_{k}'$ for $k=1,3$,
where $\supp (w_{1}')\cap \supp (w_{3}')=\emptyset
$. Then $w_{3}w_{1}'=w_{1}w_{3}'$ as words and, in
$S$, we have $w_{3}w_{1}'=w_{4}w_{1}'$ and
$w_{1}w_{3}'=w_{2}w_{3}'$. So one of the words
$w_{4}w_{1}',w_{2}w_{3}'$ is in $I$. Therefore
$\sqrt{w_{4}}w_{1}'\in I$ or $\sqrt{w_{2}}w_{3}'\in
I$. Say, for example, that the former holds. Then
$\sqrt{w_{4}}w_{1}'$ is in a nontrivial relation in
$S$.  But it is easy to see that
$\sqrt{w_{4}}w_{1}'$ cannot have $w_{i}$ as a
subword for every $i=1,2,3,4$. This contradiction
establishes assertion (3.d).

To prove (3.b) and (3.c) in this case, suppose for example that
$\ssupp(w_{2})=\emptyset$ and $\ssupp (w_{3})\neq \emptyset \neq
\ssupp (w_{4})$.  An argument as before shows that $w_{4}\prec
w_{3}\in I$ and $w_{4}\not \in I$. Hence $w_{3}\neq
\sqrt{w_{3}}\in I$. Then $\sqrt{w_{3}}=w_{1}x$ for a  word $x$.
The only relation in which $w_{1}x$ can occur must be of the form
$w_{1}x=w_{2}x$, whence we have $w_{2}\prec w_{1}$. Write
$w_{3}=v_{1}v_{3}$ where $\supp(v_{1})=\supp(w_{1})$ and
$\supp(v_{3})\cap \supp(w_{1})=\emptyset$. Let $k\geq 1$ be
minimal such that $w_{3}$ divides $w_{1}^{k}v_{3}$. Then
$w_{1}^{k}v_{3}=w_{3}y$ for a subword $y$ of $w_{1}$ such that
$y\neq w_{1}$. So, in $S$, we get
$w_{2}^{k}v_{3}=w_{1}^{k}v_{3}=w_{3}y=w_{4}y$. Since the word
$w_{4}y$ is not divisible by $w_{1},w_{2},w_{3}$ and $w_{4}\neq
\sqrt{w_{4}}$, it follows that $\sqrt{w_{4}y}\not \in I$, whence
$w_{2}^{k}v_{3}\in I$. Then $w_{2}v_{3}\in I$. But the only
relation containing this word is $w_{2}v_{3}=w_{1}v_{3}$. Since
$w_{2}v_{3}\prec w_{1}v_{3}$, we get a contradiction. We have
shown that (3.b),(3.c) are satisfied.

Finally, consider the case where there are at
exactly two overlaps between the supports of $w_{i},
i=1,2,3,4$. We may assume that $\supp (w_{1})\cap
\supp(w_{3})\neq \emptyset $ and $\supp (w_{2})\cap
\supp(w_{3})\neq \emptyset $. So $\supp (w_{4})\cap
\supp (w_{j})=\emptyset $ for every $j\neq 4$.

Suppose that $\ssupp(w_{4})\neq \emptyset$. Let $w_{1}=ab,\;
w_{2}=cd, \; w_{3}=a'c'e$, where $\supp(a)=\supp(a'), \;
\supp(c)=\supp(c')$ and the remaining factors have pairwise
disjoint supports. Let $a_{0},a_{0}'$ be words of minimal length
such that $aa_{0}=a'a_{0}'$. Clearly, $a_{0}'$ is not divisible by
$a$ and $\supp(a_{0}')\cap \supp(a_{0})=\emptyset$.

Now $abc'e=cdc'e$ in $S$ and $a'c'eb=w_{4}b$ in $S$. So
$cdc'ea_{0}=w_{4}ba_{0}'$ in $S$ and hence one of these words is
in $I$. If $w_{4}ba_{0}'$ is in $I$ then $\sqrt{w_{4}ba_{0}'}\in
I$, which is not possible because $\sqrt{w_{4}ba_{0}'}$ cannot be
rewritten in $S$ (as $\sqrt{w_{4}}$ is a proper subword of $w_{4}$
with support independent of $w_{1},w_{2},w_{3}$ and $ba_{0}'$ is
not divisible by any of $w_{1},w_{2},w_{3}$). Hence $cdc'ea_{0}\in
I$. Then $cdea_{0}\in I$ because $I$ is square free. But the only
way to rewrite $cdea_{0}$ in $S$ is $cdea_{0} =abea_{0}$. Hence
$abea_{0}\prec cdea_{0}$, so also $w_{1}=ab\prec cd=w_{2}$.
However, repeating the above argument with the roles of $w_{1}$
and $w_{2}$ switched, we also get $w_{2}\prec w_{1}$, a
contradiction. We have proved that $w_{4}$ is square free, so
(3.d) holds, and (3.c) also holds.

It remains to prove condition (3.b). Suppose that
$w_{1},w_{2}$ are not square free. By symmetry, we
may assume that $w_{2}\in I$. Then $\sqrt{w_{1}}\in
I$ and in particular the word $\sqrt{w_{1}}$ it must
be divisible by $w_{3}$. But $\supp(w_{2})\cap
\supp(w_{3})\neq \emptyset$ by the assumption, so
$\supp(w_{2})\cap \supp(w_{1})\neq \emptyset$, a
contradiction. This completes the proof of the fact
that (3) is a consequence of (2).

Now we prove (3) implies (1).  So, suppose that
the four properties (3.a)-(3.d) hold. We claim
that if $S$ is embedded in a group then the group
$SS^{-1}$ is torsion free,  and thus $S$ is a
positive affine semigroup. Note that in this
case, $SS^{-1}$ actually is a free abelian group
of rank $n-2$. Indeed, because of the assumptions
there exists $u_{i}$ and $\epsilon \in \{ 1,2\}$
so that $u_{i}\in \supp (w_{\epsilon})$ and
$\ssupp (w_{\epsilon})=\emptyset$. Re-numbering
the generators, if necessary, we may assume that
$i=1$. Then the relation $w_{1}=w_{2}$ implies
that $u_{1}=wv^{-1}$ for some $w,v\in S$ with
$\supp (w)\cup \supp (v) \cup \{ u_{1} \} =\supp
(w_{1})\cup \supp (w_{2})$, $u_{1}\not\in \supp
(w)\cup \supp (v)$ and $\supp (w)\cap \supp
(v)=\emptyset$. It follows that $$SS^{-1}=\gr (
u_{2},\ldots , u_{n}\mid
w_{3}(wv^{-1},u_{2},\ldots , u_{n})
=w_{4}(wv^{-1},u_{2},\ldots , u_{n})) .$$ If the
second property of (3.d)  holds then $\supp
(w_{4})\cap (\bigcup_{i=1}^{3} \supp (w_{i}))
=\emptyset$ and $\ssupp (w_{4}) =\emptyset$. So,
in particular, $u_{1}\not\in \supp (w_{4})$ and
for $u_{k}\in \supp (w_{4})$ we have that
$u_{k}\not\in \supp (w)\cup \supp (v) \cup \sup
(w_{3})$ and $$u_{k} = w_{3}
(wv^{-1},u_{2},\ldots , u_{n}) u^{-1}$$ with
$w_{4}=uu_{k}$ and $\supp (w_{4})=\supp (u) \cup
\{ u_{k}\}$. Hence we obtain that $SS^{-1}=\gr
(\{ u_{2},\ldots , u_{n}\} \setminus \{ u_{k}\}
)$ and this is a free abelian group of rank
$n-2$, as claimed. If, on the other hand, the
first property of (3.d)  holds then, without loss
of generality, we may assume that $\supp
(w_{1})\cap \supp (w_{3})\neq \emptyset$, $\ssupp
(w_{2})= \emptyset$  and $u_{1}\in \supp
(w_{2})$. So, $u_{1}\not\in \supp (w_{3})$. If
$\ssupp (w_{3})=\emptyset$ then choose $u_{k}\in
\supp (w_{3})$ and write $w_{3}=u_{k}v'$ with
$u_{k}\not\in \supp (v')$ and $\supp (w_{3})=\{
u_{k}\} \cup \supp (v')$. So
$u_{k}=w_{4}(v')^{-1}$. Note that $u_{1}\not\in
\supp (w_{4})\cup \supp (v')$. It follows that
$SS^{-1}=\gr (\{ u_{2},\ldots ,u_{n}\} \setminus
\{ u_{k}\})$, a free abelian group of rank $n-2$.
Finally, if $\ssupp (w_{3})\neq \emptyset$ then
$\ssupp (w_{4})=\emptyset$. In this case write
$w_{4}=u_{l}v''$ for some $v''$ with
$u_{l}\not\in \supp (v'')$ and $\supp (w_{4})=\{
u_{k} \} \cup \supp (v'')$. It follows that
$SS^{-1} =\gr (\{ u_{2},\ldots , u_{n}\}
\setminus \{ u_{l}\})$,  again a free abelian
group of rank $n-2$, as desired.

So now we show that $S$ is cancellative and thus
embedded in $\Fa_{n-2}$. By symmetry we can assume
that $\ssupp(w_{4}) = \emptyset$. Then write $$w_{2}
= y_{1}^{\gamma_{1}}\cdots y_{q}^{\gamma_{q}},
 \quad w_{4} = x_{1}\cdots x_{p-1}x_{p},$$ $\gamma_{i} \geq
1$, where $x_{1},\ldots ,x_{p},y_{1},\ldots
,y_{q}\in \{ u_{1},\ldots , u_{n}\}$, and
$\supp(w_{4})$ does not intersect nontrivially the
support of any other word in the defining relations.

Let $F$ be the  free abelian monoid with basis
$\supp(w_{1}) \cup
 \{y_{1},\ldots, y_{q}\} \cup \supp(w_{3}) \cup \{x_{1},\ldots ,x_{p-1}\} $.
Then let $T = F / \rho$, where $\rho$ is the congruence defined by
the relation $w_{1} = w_{2}$. Since $\ssupp(w_{1}) = \emptyset$ or
$\ssupp(w_{2}) = \emptyset$, we know from
Proposition~\ref{1relation} that $T$ is a normal positive monoid.
In particular, $TT^{-1}$ is a torsion free group. Consider the
semigroup morphism
 $$f: T \times \langle u \rangle
 \longrightarrow TT^{-1}$$
defined by $f(t) = t$, for $t\in T$ and
$f(u)=w_{3}z^{-1}$ and $z = x_{1}\cdots x_{p-1}$.
Note that $f(w_{3}) = f(zu)$. Hence the above
morphism induces the following natural morphisms
 $$T \times \langle u \rangle \stackrel{\pi}{\longrightarrow} (T \times
 \langle u \rangle ) / \nu \stackrel{\overline{f}}{\longrightarrow} TT^{-1}
 ,$$
with $\nu$ the congruence defined by the relation
 $w_{3} = zu$.
Put $M = (T \times \langle u \rangle )/ \nu$ and
note that $$M \cong S.$$ For simplicity we denote
$\pi(t)$ as $\overline{t}$, for $t \in T\times
\langle u \rangle$. We note that $\pi_{\mid_{T}}$,
the restriction of $\pi$ to $T$, is injective.
Indeed, suppose $s,t \in T$ are such that $\pi(s) =
\pi(t)$. Then $$s - t \in K[T \times \langle u
\rangle] (zu - w_{3}),$$ an ideal in $K[T \times
\langle u \rangle]$. So, $s - t =\alpha (zu -
w_{3})$, for some $\alpha \in K[T \times \langle u
\rangle]$. Now $K[T \times \langle u \rangle]$ has a
natural $\N$-gradation, with respect to the degree
in $u$. Clearly, $s - t$ and $w_{3}$ have degree
zero. Let $\alpha_{h}$ be the highest degree term of
$\alpha$ with respect to this gradation. Then, $$0 =
\alpha_{h} z u.$$ Since $T \times \langle u \rangle$
is contained in a torsion free group, we know that
$K[T \times \langle u \rangle]$ is a domain. So we
get that $\alpha_{h} = 0$ and thus $\alpha = 0$.
Hence $s = t$ and therefore indeed $\pi_{\mid_{T}}$
is injective. So we will identify the element
$\pi(t)$ with $t$, for $t \in T$.

Next we note that $\overline{u}$ is a cancellable
element in $M$. Indeed, let $\overline{x},
\overline{y} \in M$ and suppose $\overline{u} \
\overline{x} = \overline{u} \ \overline{y}$. This
means that $$ux-uy \in K[T \times \langle u \rangle]
(uz - w_{3}),$$ i.e.
\begin{eqnarray}\label{!}
ux - uy & = & \alpha(uz - w_{3})
\end{eqnarray}
for some $\alpha \in K[T \times \langle u \rangle]$,
where $x,y\in T\times \langle u \rangle $ are
inverse images of $\overline{x}, \overline{y}$.
Again consider the $\N$-gradation on $K[T \times
\langle u \rangle]$ via the degree in $u$. Let
$\alpha_{0}$ be the zero degree component of
$\alpha$. Then it follows that $$0 =
\alpha_{0}w_{3}.$$ Hence $\alpha_{0} = 0$, as $K[T]$
is a domain, and thus $$\alpha \in K[T\times \langle
u \rangle] u.$$ Using again that $K[T\times \langle
u \rangle]$ is a domain, we get from (\ref{!}) that
$$x - y \in K[T \times \langle u \rangle] (uz -
w_{3}).$$ Hence $\overline{x} = \overline{y} \in M$,
as desired.

In the above we thus have shown that $\overline{u}$
is cancellable in $M$. Hence $x_{p}$ is cancellable
in $S$. The argument of the proof holds for all
elements $x_{1},\ldots , x_{p}$. So, all elements
$x_{1},\ldots ,x_{p}$ are cancellable in $S$. By  a
similar argument, if $\ssupp(w_{2}) = \emptyset$,
this also holds for all elements $y_{i} \in
\supp(w_{2}) \setminus \supp(w_{3})$.

On the other hand, if $\ssupp(w_{2}) \neq \emptyset$
and thus $\ssupp(w_{1}) = \emptyset$, then similarly
one shows that $u_{i}$ is cancellable in $S$, for
every $u_{i} \in \supp(w_{1}) \setminus
\supp(w_{3})$. Clearly, $S$ is contained in its
localization $S_{C}$, with respect to the
multiplicatively closed set of the cancellable
elements. In view of the form of the defining
relations of $S$, this implies that $S_{C}$ is a
group. So $S$ is a cancellative monoid in $SS^{-1} =
\Fa_{n-2}$.

Finally, we show that $S$ is normal, by proving it is a union of
finitely many finitely generated free abelian monoids. To so, note
that conditions (3.a)-(3.d) imply that $\ssupp(w_{i})=\emptyset$
and $\ssupp(w_{j})=\emptyset$ for some $i\in \{ 1,2\}$ and $j\in
\{ 3,4\}$. Furthermore, $\supp(w_{i})\cap \supp (w_{l})=\emptyset$
for all $l$ with $l\neq i$, or $\supp (w_{j})\cap \supp
(w_{l})=\emptyset$ for all $l$ with $l\neq j$. Without loss of
generality we may assume the former holds.  Note that if $w_{k} =
u_{q}$ for some $k$ and some $q$ then the assertion follows from
Proposition~\ref{1relation}. Hence, without loss of generality, we
may assume that $|\supp (w_{k})|>1$ for $k=1,2,3,4$.

Because $\ssupp(w_{i})=\emptyset$, it is easily
seen, using the relation involving $w_{i}$, that
$s$ can be written as a product of elements of
$\{ u_{1}, \ldots , u_{n}\} \setminus \{u\}$ for
some $u\in \supp (w_{i})$. If not all elements of
$\supp (w_{j})$ occur in this product of $s$,
then $s\in \langle \{ u_{1},\ldots , u_{n}\}
\setminus \{ u,v\}\rangle$, with $v\in \supp
(w_{j})$. Now because of the defining relations
one easily sees that $\langle \{ u_{1}, \ldots ,
u_{n}\} \setminus \{ u,v\} \rangle $ is a free
abelian monoid, as desired. If, on the other
hand, all elements of $\supp(w_{j})$ occur in the
expression of $s$ then, using the relation
involving $w_{j}$ (several times if needed) and
using the fact that $\supp (w_{i})\cap \supp
(w_{l})=\emptyset$ for all $l\neq i$, we can
reduce to the previous case. This ends the proof.
\end{proof}

As a matter of example, it follows at once from
Theorem~\ref{2relations} that the commutative
algebra $K\langle u_{1},u_{2},u_{3},u_{4},u_{5}
\mid u_{1}u_{2} = u_{3}^{2}, \; u_{1}u_{3} =
u_{4}u_{5}\rangle$ is a normal domain, while the
commutative algebra  $K\langle
u_{1},u_{2},u_{3},u_{4} \mid u_{1}u_{2} =
u_{3}^{2}, \; u_{1}u_{3} = u_{4}^{2}\rangle$ is a
domain that is not normal.

Finally, we describe the class group of positive monoid defined by
two relations. We use the same notation as in the proof of
Theorem~\ref{2relations}. If $(\supp (w_{1})\cup \supp (w_{2}))
\cap (\supp (w_{3})\cup \supp (w_{4}))=\emptyset$ then $S\cong
S_{1}\times S_{2}$, with $S_{1}=\langle \supp (w_{1})\cup
\supp(w_{2}) \mid w_{1}=w_{2}\rangle$ and $S_{2}=\langle \supp
(w_{3})\cup \supp(w_{4}) \mid w_{3}=w_{4}\rangle$. Clearly, in
this case, $\Cl(S) \cong \Cl(S_{1})\times \Cl(S_{2})$, and the
result follows from Theorem~\ref{1main}. So, assume $S$ satisfies
one of the properties in condition (3.d) in
Theorem~\ref{2relations}. Then, we can write $$S = \langle
u_{1},\ldots,u_{n},\ldots , u_{m}\rangle$$ with relations
\begin{eqnarray*}
u_{1}\cdots u_{k_{1}}u_{k_{2}+1}\cdots u_{k_{3}}&=&
u_{k_{1}+1}^{a_{k_{1}+1}}\cdots
u_{k_{2}}^{a_{k_{2}}}u_{k_{3}+1}^{a_{k_{3}+1}}\cdots
u_{k_{4}}^{a_{k_{4}}}\\ u_{1}^{a_{1}}\cdots
u_{k_{1}}^{a_{k_{1}}}u_{k_{1}+1}^{b_{k_{1}+1}}\cdots
u_{k_{2}}^{b_{k_{2}}}u_{k_{4}+1}^{a_{k_{4}+1}}\cdots
u_{k_{5}}^{a_{k_{5}}} &=& u_{k_{5}+1}\cdots u_{n},
\end{eqnarray*}
with $0<k_{1} \leq k_{2} \leq k_{3} \leq k_{4} \leq k_{5} < n \leq
m$ and all $a_{i},b_{j}\geq 1$ and (we agree that if
$k_{1}=k_{2}$, $k_{2}=k_{3}$, $k_{3}=k_{4}$ or $k_{4}=k_{5}$ then
the factors $u_{k_{1}+1}^{a_{k_{1}+1}} \cdots
u_{k_{2}}^{a_{k_{2}}}$, $u_{k_{1}+1}^{b_{k_{1}+1}} \cdots
u_{k_{2}}^{b_{k_{2}}}$, $u_{k_{2}+1}\cdots u_{k_{3}}$,
$u_{k_{3}+1}^{a_{k_{3}+1}}\cdots u_{k_{4}}^{a_{k_{4}}}$, or
$u_{k_{4}+1}^{a_{k_{4}+1}}\cdots u_{k_{5}}^{a_{k_{5}}}$ are the
empty words). So, the two cases discussed in condition (3.d) of
Theorem~\ref{2relations} correspond to $k_{1}=k_{2}$ and
$k_{1}<k_{2}$, respectively.

As in the previous section, in order to compute the class group,
we also may assume that $n=m$. Moreover, we may assume that
$w_{i}\not \in \{ u_{1},\ldots ,u_{n}\}$ for $i=1,2,3,4$, as
otherwise $S$ can be presented by a single relation and then the
class group is given in Theorem~\ref{1main}. Under this
restriction, in the next lemma, we describe the principal ideals
as divisorial products of minimal prime ideals. Note that there
are two possible types of minimal primes in $S$. First, there are
$$Q = (u_{i},u_{j}),$$ where $u_{i}$ and $u_{j}$ each belong to
the support of different sides of one of the defining relations
and do not belong to the supports of the words in the other
relation. To prove that $Q$ is a prime ideal we may assume, by
symmetry, that $u_{i},u_{j}\in \supp (w_{1})\cup \supp (w_{2})$.
Clearly, $S/Q$ is then generated by the natural images of the
elements $u_{q}, q\neq i,j$, subject to the unique relation
$w_{3}=w_{4}$. Since $u_{i},u_{j}\not \in \supp (w_{3})\cup \supp
(w_{4})$, it is easily seen that $(S/Q )\setminus \{ 0\}$ is a
multiplicatively closed set, as desired. Second, there are minimal
primes of the form
 $$Q = (u_{i},u_{j},u_{k}),$$
where $u_{i}$ belongs to the support of a word in each of the two
relations, $u_{j}$ and $u_{k}$ belong to the support of a word in
a defining relation but on a different side than $u_{i}$, and
furthermore $u_{j}$ and $u_{k}$ are involved in different
relations.  In particular, $j\neq k$. Clearly, existence (and the
number) of minimal primes of the latter type depends on the
existence of strict inequalities $k_{i}<k_{i+1}$.

The formulas obtained in the following Lemma~\ref{prin2} should be
interpreted in such a way that principal ideals $Su_{w}$ and
primes $P_{y,z}$ or $P_{t,v,x}$ are deleted if some index does not
occur in the defining relations. So, for example $P_{y,k_{3}+1}$
is not defined and hence ignored if $k_{3}=k_{4}$.

\begin{lemma}\label{prin2}
Let
\begin{eqnarray*} S = \langle u_{1},\ldots,u_{n}
&\mid& u_{1}\cdots u_{k_{1}}u_{k_{2}+1}\cdots
u_{k_{3}}= u_{k_{1}+1}^{a_{k_{1}+1}}\cdots
u_{k_{2}}^{a_{k_{2}}}u_{k_{3}+1}^{a_{k_{3}+1}}\cdots
u_{k_{4}}^{a_{k_{4}}}\\ && u_{1}^{a_{1}}\cdots
u_{k_{1}}^{a_{k_{1}}}u_{k_{1}+1}^{b_{k_{1}+1}}\cdots
u_{k_{2}}^{b_{k_{2}}}u_{k_{4}+1}^{a_{k_{4}+1}}\cdots
u_{k_{5}}^{a_{k_{5}}} = u_{k_{5}+1}\cdots
u_{n}\rangle ,
\end{eqnarray*}
with  $0<k_{1} \leq k_{2} \leq k_{3} \leq k_{4} \leq k_{5} < n$
and all $a_{i},b_{j}\geq 1$, be a normal monoid that cannot be
presented with a single relation. Put $P_{y,z}$, the  minimal
prime ideal of $S$ generated by $\{u_{y},u_{z}\}$, $y \in
\{k_{2}+1,\ldots , k_{3}\}$, $z\in \{k_{3}+1, \ldots , k_{4}\}$ or
$y \in \{k_{4}+1,\ldots , k_{5}\}$, $z\in \{k_{5}+1, \ldots , n\}$
and put $P_{t,v,x}$, the minimal prime ideal of $S$ that is
generated by $\{ u_{t},u_{v},u_{x}\}$, $t \in \{1,\ldots ,
k_{1}\}$, $v \in \{k_{1}+1,\ldots , k_{2}, k_{3}+1, \ldots
,k_{4}\}$, $x \in \{k_{5}+1,\ldots , n\}$ or $t \in
\{k_{2}+1,\ldots , k_{3}\}$, $v \in \{k_{1}+1,\ldots , k_{2}\}$,
$x \in \{k_{5}+1,\ldots , n\}$. Then
\begin{enumerate}
\item
$Su_{w}  =
\left(\prod_{l=k_{5}+1}^{n}\left(\prod_{m=1}^{k_{1}}
P_{m,w,l} \;
\prod_{m=k_{2}+1}^{k_{3}}P_{m,w,l}\right)\right)^{*}$,
for  $w\in \{k_{1}+1  , \ldots , \\ k_{2} \}$,
\item
$
Su_{w} =
\left(\prod_{l=k_{5}+1}^{n}\left(\prod_{m=1}^{k_{1}}
P_{m,w,l}\right) \right)^{*} \; *\; \left(
\prod_{m=k_{2}+1}^{k_{3}}P_{m,w}\right)^{*}$, for
$w\in \{k_{3}+1, \ldots , k_{4}\}$,
\item
$Su_{w} = \left(\prod_{l=k_{5}+1}^{n}
P_{w,l}\right)^{*}$,  for $w\in \{k_{4}+1, \ldots ,
k_{5}\}$,
\item
$
Su_{w} =
\left(\prod_{l=k_{5}+1}^{n}\left(\prod_{m=k_{1}+1}^{k_{2}}
P_{w,m,l}^{a_{m}} \;
\prod_{m=k_{3}+1}^{k_{4}}P_{w,m,l}^{a_{m}}\right)\right)^{*}$,
for $w \in \{1,\ldots , \\k_{1}\}$,
\item
$
Su_{w} =
\left(\prod_{l=k_{5}+1}^{n}\left(\prod_{m=k_{1}+1}^{k_{2}}
P_{w,m,l}^{a_{m}}\right) \right)^{*} \; * \; \left(
\prod_{m=k_{3}+1}^{k_{4}}P_{w,m}^{a_{m}}\right)^{*}$,
for $w \in \{k_{2}+1,\ldots , k_{3}\}$,
\item
$
Su_{w} = \left(
\prod_{l=1}^{k_{1}}\left(\prod_{m=k_{1}+1}^{k_{2}}
P_{l,m,w}^{a_{m}} \;
\prod_{m=k_{3}+1}^{k_{4}}P_{l,m,w}^{a_{m}}\right)^{a_{l}}\right)^{*}
\ast\\ \quad
\left(\prod_{m=k_{1}+1}^{k_{2}}\left(\prod_{l=1}^{k_{1}}P_{l,m,w}
\;\prod_{l=k_{2}+1}^{k_{3}}P_{l,m,w}\right)^{b_{m}}\right)^{*}
\; *\;  \left( \prod_{l=k_{4}+1}^{k_{5}}
P_{l,w}^{a_{l}}\right)^{*}$,\\ for $w \in
\{k_{5}+1,\ldots , n\}$,
\end{enumerate}
\end{lemma}
\begin{proof}
For $w\in \{ 1,\ldots, n\}$, one notices that in the expressions
for $S_{w}$, in the statement of the lemma, precisely all the
minimal primes $P$ occur that contain $u_{w}$. Using the defining
relations one then easily verifies, as in the proof of
Lemma~\ref{prin}, that the proposed formulae hold in the
localizations $S_{P}$. Hence the result follows.
\end{proof}

Our next aim is to describe the class group of $S$.
Surprisingly, the proof is obtained by a reduction
to the case considered in Theorem~\ref{1main}. The
definitions of  $d_{1}$ and $d_{2}$ in the following
result should again be interpreted in the correct
way when some $k_{i}=k_{i+1}$. We agree to ignore
all $a_{i}$ (respectively, $b_{j}$) for which
$u_{i}$ (respectively $u_{j}$) does not occur in the
defining relations.

\begin{theorem}\label{2main}
Let
\begin{eqnarray*} S = \langle u_{1},\ldots,u_{n},\ldots , u_{m} \mid
u_{1}\cdots u_{k_{1}}u_{k_{2}+1}\cdots u_{k_{3}}&=&
u_{k_{1}+1}^{a_{k_{1}+1}}\cdots
u_{k_{2}}^{a_{k_{2}}}u_{k_{3}+1}^{a_{k_{3}+1}}\cdots
u_{k_{4}}^{a_{k_{4}}}\\ u_{1}^{a_{1}}\cdots
u_{k_{1}}^{a_{k_{1}}}u_{k_{1}+1}^{b_{k_{1}+1}}\cdots
u_{k_{2}}^{b_{k_{2}}}u_{k_{4}+1}^{a_{k_{4}+1}}\cdots
u_{k_{5}}^{a_{k_{5}}} &=& u_{k_{5}+1}\cdots
u_{n}\rangle
\end{eqnarray*}
(with $0<k_{1} \leq k_{2} \leq k_{3} \leq k_{4} \leq k_{5} < n\leq
m$ and all $a_{i},b_{j}\geq 1$)
 be a normal positive monoid that does not admit a presentation with a
 single defining relation.  Let
 $Q=\{a_{t}a_{v}+b_{v}  \mid t\in \{ 1,\ldots, k_{1} \}, v\in
\{ k_{1}+1,\ldots ,k_{2}\} \} \cup \{ a_{t}a_{v}  \mid t\in \{
1,\ldots, k_{1} \}, v\in \{ k_{3}+1,\ldots ,k_{4}\} \} \cup\{
a_{y} \mid y\in \{ k_{4}+1,\ldots ,k_{5} \} \}$.
 Then $$\Cl(K[S]) \cong \Cl(S)
\cong \Z^{f}\times
(\Z_{d_{1}})^{k_{1}+k_{3}-k_{2}-1} \times
(\Z_{d_{2}})^{n-k_{5}-1},$$ where
\begin{eqnarray*}
f &=& (k_{3} - k_{2})(k_{4}-k_{3})+ (k_{5} -
k_{4})(n - k_{5}) + k_{1}(k_{4} - k_{3} +
k_{2}-k_{1})(n-k_{5})\\ && + \; (k_{3} -
k_{2})(k_{2}-k_{1})(n-k_{5}) -(n-2),
\end{eqnarray*} with $$
d_{1} = \gcd(a_{k_{1}+1},
\ldots , a_{k_{2}}, a_{k_{3}+1}, \ldots , a_{k_{4}})$$ and
$$d_{2} = \left\{ \begin{array}{ll}
 \gcd(a_{1}d_{1}, \ldots , a_{k_{1}}d_{1}, b_{k_{1}+1}, \ldots ,
b_{k_{2}}, a_{k_{4}+1}, \ldots , a_{k_{5}}) &  \mbox{ if } \,
k_{2}<k_{3} \\
 \gcd( q \mid q\in Q) & \mbox{ if } \, k_{2}=k_{3} \end{array} \right. .$$
\end{theorem}
\begin{proof} As mentioned earlier, withou loss of generality we may assume
that $n=m$. It is shown in the proof of Theorem~\ref{2relations}
that $SS^{-1} \cong \Fa_{n-2}$, the free abelian group of rank
$n-2$. Because $\U(S) = \{1\}$, we get that $P(S)$ and $SS^{-1}$
are isomorphic, and thus they have the same torsion free rank.
Since the torsion free rank of $\Cl(S)$ is the difference of the
torsion free rank of $D(S)$ and the torsion free rank of $P(S)$,
to establish the description of the torsion free part of $\Cl(S)$,
we only need to show that there are $(k_{3} - k_{2})(k_{4}-k_{3})+
(k_{5} - k_{4})(n - k_{5}) + k_{1}(k_{4} - k_{3} +
k_{2}-k_{1})(n-k_{5}) + (k_{3} - k_{2})(k_{2}-k_{1})(n-k_{5})$
minimal primes in $S$. But this easily follows from the
description of the minimal primes given Lemma~\ref{prin2}.

As in the proof of Theorem~\ref{1main}, we consider
$\Cl (S)$ as a finitely generated $\mathbb
Z$-module, so that its presentation is determined by
an integer matrix $M$ of size $r\times n$, where $r$
is the number of minimal primes in $S$, hence the
basis of $D(S)$. Therefore, the rows are indexed by
all triples $(t,v,x)$ and all pairs $(y,z)$, as
described in Lemma~\ref{prin2}. We agree on the
following ordering of the set of rows of $M$: all
triples $(t,v,x)$ are ordered lexicographically, so
are all the pairs $(y,z)$ and $(t,v,x)<(y,z)$ for
every $t,v,x, y,z$. The columns are indexed by
$1,2,\ldots, n$, where the $i$-th column corresponds
to the generator $Su_{i}$, written as a vector in
terms of the minimal primes of $S$, as in
Lemma~\ref{prin2}.

We consider the block decomposition of $M$
determined by the following partitions of the sets
$C$ and $R$ of columns and rows: $$C=D_{1}\cup
D_{2},$$ where $D_{1}=\{ 1,\ldots ,k_{5} \}$ and
$D_{2}=\{ k_{5}+1,\ldots ,n\}$. Notice that
$|D_{2}|\geq 2$  because $S$ does not admit a
presentation with one defining relation. Let $R_{0}=
\{ (y,z) \mid y\in \{ k_{2}+1, \ldots ,k_{3}\}, z
\in \{ k_{3}+1,\ldots ,k_{4}\} \}$, $R_{y}=\{ (y,z)
\mid z\in \{ k_{5}+1, \ldots ,n\} \}$ for $y\in \{
k_{4}+1,\ldots ,k_{5}\}$. For every triple $(t,v,x)$
we also define $R_{t,v}=\{ (t,v,x) \mid x\in
\{k_{5}+1,\ldots ,n\} \}$. Then $$R=\bigcup
R_{t,v}\cup R_{0}\cup \bigcup
_{y=k_{2}+1}^{k_{3}}R_{y},$$ where the first union
runs over all pairs $(t,v)$ such that the set $R$ of
rows contains a triple of the form $(t,v,x)$.

Consider any of the block submatrices $M_{R_{t,v},C}$ or
$M_{R_{y},C}$, with $R_{t,v},R_{y}$ as above. From
Lemma~\ref{prin2} it follows that, ignoring the zero columns of
this submatrix, it has the form
$$\left( \begin{array}{cccccc}
a &  b & d & 0 & \cdots & 0 \\
a & b  &  0 & d & \cdots & 0 \\
\cdots & \cdots & \cdots& \cdots& \cdots& 0 \\
a & b & 0 & \cdots & 0 & d
\end{array} \right) ,$$
for some $a,b$ such that either $a=1$ or $b=1$ and
for some $d$. Here the columns of the scalar matrix
determined by $d$ are indexed by $D_{2}$. So,
subtracting the first row in each such block
($M_{R_{t,v},C}$ or $M_{R_{y},C}$) from all the
remaining rows in this block and next subtracting
the last $n-(k_{5}+1)$ columns of the entire matrix
from column $k_{5}+1$, we get a matrix $M'$ such
that each block $M'_{X,D_{1}}$, for $X=R_{y}$ or
$X=R_{t,v}$, has only the first row nonzero and
$M'_{R,D_{2}}=M_{R,D_{2}}$. Moreover
$M'_{R_{0},C}=M_{R_{0},C}$. Therefore, the nonzero
entries of the last column of $M'$ are the only
nonzero entries in their respective rows. Denote by
$Y$ the set of all such  rows of $M'$. Then these
nonzero entries (in the last column of $M'$),  and
with our convention as explained before the theorem,
are: $$\begin{array}{cclll} a_{t}a_{v}+b_{v} &
\mbox{ for } & t\in \{ 1,\ldots, k_{1} \}, v\in \{
k_{1}+1,\ldots ,k_{2}\} & \mbox{ if } & k_{1}\neq
k_{2} \\ a_{t}a_{v} & \mbox{ for } & t\in \{
1,\ldots, k_{1} \}, v\in \{ k_{3}+1,\ldots ,k_{4}\}
& \mbox{ if } & k_{3}\neq k_{4} \\ b_{t} & \mbox{
for } & t\in \{ k_{1}+1,\ldots ,k_{2} \} & \mbox{ if
} & k_{1}\neq k_{2}, k_{2}\neq k_{3}\\ a_{y} &
\mbox{ for } & y\in \{ k_{4}+1,\ldots ,k_{5} \} &
\mbox{ if } & k_{4}\neq k_{5} .
\end{array} $$
 Notice that the greatest common divisor of the
specified set of elements is equal to $d_{2}$, as
defined in the statement of the theorem. Thus, row
elimination within the block $M'_{Y,C}$ allows us to
produce a row of the form $(0,\ldots ,0, d_{2})$ and
replace all other rows of $M'_{Y,C}$ by zero rows.
The same argument can be applied to the nonzero
entries in the subsequent columns: $n-1, n-2,
\ldots,k_{5}+2$. This leads to a matrix $M''$ (of
the same size as the original matrix $M$) with
$n-k_{5}-1$ rows of the form $(0,\ldots
,d_{2},0,\ldots ,0)$, with $d_{2}$ in positions
$k_{5}+2,\ldots ,n$, with no other nonzero entries
in their respective columns. So, it remains to find
the normal form of the matrix $N$ obtained by
deleting in $M''$ the last $n-k_{5}-1$ columns and
the rows that contain the nonzero entries in these
columns. It is easy to see that the last column of
$N$ is a $\mathbb Z$-combination of the remaining
columns. Namely, we have
$C_{k_{5}+1}=a_{1}C_{1}+\cdots +a_{k_{1}}C_{k_{1}}
+b_{k_{1}+1}C_{k_{1}+1} +\cdots +
b_{k_{2}}C_{k_{2}}+a_{k_{4}+1}C_{k_{4}+1}+\cdots +
a_{k_{5}}C_{k_{5}}.$   Hence by column operations we
can make this column zero. Then, deleting this
column, we get a matrix with $k_{4}$ columns that is
of he form $\left(\begin{array}{cc} N' & 0
\\ 0 & I
\end{array} \right)$ for a matrix $N'$ and the
identity $(k_{5}-k_{4})\times (k_{5}-k_{4})$-matrix
$I$. It is easy to see that $N'$ corresponds to the
monoid $T$ with the presentation $u_{1}\cdots
u_{k_{1}}u_{k_{2}+1}\cdots
u_{k_{3}}=u_{k_{1}+1}^{a_{k_{1}+1}}\cdots
u_{k_{2}}^{a_{k_{2}}}u_{k_{3}+1}^{a_{k_{3}+1}}\cdots
u_{k_{4}}^{a_{k_{4}}}$ and with the generating set
$u_{1},\ldots ,u_{k_{4}}$.  Hence, by
Theorem~\ref{1main}, $\Cl (T) ={\mathbb Z}^{e}
\times {\mathbb Z}_{d_{1}}^{k_{3}-k_{2}+k_{1}-1}$,
where $e=
(k_{1}+k_{3}-k_{2})(k_{2}-k_{1}+k_{4}-k_{3})-(k_{4}-1)$.
Therefore, the normal form of $M$ has
$k_{3}-k_{2}+k_{1}-1$ copies of $d_{1}$ and
$n-k_{5}-1$ copies of $d_{2}$ and a certain number
of entries equal to $1$. By the comment at the
beginning of the proof, it must have $f$ zero rows.
Hence, the result follows.
\end{proof}

{\bf Acknowledgments} This research was supported by
the Onderzoeksraad of Vrije Universiteit Brussel,
Fonds voor Wetenschappelijk Onderzoek (Flanders),
Flemish-Polish bilateral agreement BIL2005/VUB/06
and a MNiSW research grant N201 004 32/0088
(Poland). The first author was also funded by a Ph.D
grant of the Institute for the Promotion of
Innovation through Science and Technology in
Flanders (IWT-Vlaanderen).

The authors are grateful to the referee for making several
valuable comments and suggestions. This resulted in a completely
revised format of an earlier version of the paper.

\vspace{20pt}

 \noindent
 \begin{tabular}{ll}
 I. Goffa and E. Jespers & J. Okni\'{n}ski\\
 Department of Mathematics& Institute of Mathematics\\
 Vrije Universiteit Brussel & Warsaw University\\
 Pleinlaan 2& Banacha 2\\
 1050 Brussel, Belgium& 02-097 Warsaw, Poland\\
 efjesper@vub.ac.be and igoffa@vub.ac.be & okninski@mimuw.edu.pl
 \end{tabular}

\end{document}